\documentclass[12pt]{article}
\usepackage{amsmath}
\usepackage{amsthm}
\usepackage{amssymb}
\usepackage{graphics}
\usepackage{enumerate}   
\newcommand{\be}{\begin{equation}}
\newcommand{\ee}{\end{equation}}
\newcommand{\ben}{\begin{eqnarray*}}
\newcommand{\een}{\end{eqnarray*}}
\newtheorem{examp}{\sc Example}
\newtheorem{remk}{\sc Remark}
\newtheorem{corol}{\sc Corollary}
\newtheorem{lemma}{\sc Lemma}
\newtheorem{theorem}{\sc Theorem}
\newtheorem{defn}{\sc Definition}
\newcommand{\bt}{\begin{theorem}}
\newcommand{\et}{\end{theorem}}
\newcommand{\bl}{\begin{lemma}}
\newcommand{\el}{\end{lemma}}
\newcommand{\bed}{\begin{defn}}
\newcommand{\eed}{\end{defn}}
\newcommand{\brem}{\begin{remk}}
\newcommand{\erem}{\end{remk}}
\newcommand{\bex}{\begin{examp}}
\newcommand{\eex}{\end{examp}}
\newcommand{\bcl}{\begin{corol}}
\newcommand{\ecl}{\end{corol}}

\newcommand{\NI}{\noindent}

\newcommand{\al}{\alpha}

\newcommand{\vsp}{\vskip 0.5em}

\theoremstyle{definition}
\theoremstyle{remark}

\numberwithin{equation}{section}
\numberwithin{theorem}{section}
\numberwithin{lemma}{section}

\begin{document}
\title {More On  Hidden $Z$-Matrices and Linear Complementarity Problem }
\author{R. Jana$^{a}$, A. Dutta$^{a, 1}$, A. K. Das$^{b}$\\
\emph{\small $^{a}$Jadavpur University, Kolkata , 700 032, India.}\\	
\emph{\small $^{b}$Indian Statistical Institute, 203 B. T.
	Road, Kolkata, 700 108, India.}\\
\emph{\small $^{1}$Email: aritradutta001@gmail.com}}
\date{}
\maketitle

\date{}
\maketitle
\begin{abstract}
	\NI In this article we study linear complementarity problem with hidden $Z$-matrix. We extend the results of Fiedler and Pt{\'a}k for the linear system in complementarity problem using game theoretic approach. We establish a result related to singular hidden $Z$-matrix. We show that for a non-degenerate feasible basis, linear complementarity problem with hidden $Z$-matrix has unique non-degenerate solution under some assumptions. The purpose of this paper is to study some properties of hidden $Z$-matrix.
 \noindent \\

\NI{\bf Keywords:} Hidden $Z$-matrix, linear complementary problem, value of a matrix, $P_0$-matrix, $E_0$-matrix.\\ 

\NI{\bf AMS subject classifications:} 90C33, 15A39, 15B99.

\end{abstract}
\footnotetext[1]{Corresponding author}
\section{Introduction}
 The linear complementarity problem finds a vector in a finite dimensional real vector space that satisfies a particular system of inequalities. The problem is defined as follows: 
 
 Given $A\in R^{n\times n}$ and a vector $\,q\,\in\,R^{n},\,$ we consider the linear complementarity problem as
 	\be \label{1}
 	q + Az \geq 0, \, z \geq 0,
 	\ee
 	\be \label{2}
 	z^T(q + Az) = 0. \ee
More specifically the problem is denoted as LCP$(q, A)$ for which we find a vector $z \in R^n$ satisfying the linear system \ref{1} as well as the complementary condition \ref{2} or show that there does not exist any $z$ satisfying the linear system \ref{1} as well as the complementary condition \ref{2}. The case of LCP$(q, A)$ with $q = 0$ is known as homogeneous LCP with $A$ for which $z$ in the solution set of LCP$(0, A)$ implies that $\lambda z$ belongs to the solution set of LCP$(0, A)$ for all scaler $\lambda \geq 0.$ For details see \cite{das1}. The linear complementarity problem includes linear programming, linear fractional programming, convex quadratic programming, bimatrix game and a number of applications reported in operations research, multiple objective programming problem, mathematical economics, geometry and engineering.
 
The linear complementarity problem along with a hidden $Z$-matrices received wide attention in the literature. A generalization of $Z$-matrices was addressed by Mangasarian \cite{mangasarian} to study the linear complementarity problems solvable as linear program. Pang \cite{pang} proposed this class as hidden $Z$-matrices. A matrix $A \in R^{n \times n}$ is called hidden $Z$-matrix if there exist Z-matrices $X,$ $Y \in R^{n \times n}$ and $r, s \in R^{n}_{+}$ such that
 		\be
 		 AX = Y,
 		 \ee
 		 \be \label{hidden}
 		 r^TX + s^TY >0.
 	    \ee

A hidden $Z$-matrix is said to be completely hidden $Z$-matrix if all its principal submatrices are hidden $Z$-matrix. Fiedler and Pt{\'a}k \cite{fiedler} studied $Z$-matrix in the context of linear complementarity problem. They showed that existence of a strictly positive vector $x$ for a $Z$-matrix $A$ such that $Ax \geq 0$ allows $A$ to be $P_0$-matrix.  In this paper we extend this result in terms of hidden $Z$-matrix. The paper is organized as follows. Section 2 presents some basic notations and results. In section 3, we show a hidden $Z$-matrix under some additional conditions to be a $P_0$ matrix. We settle a result related to singular hidden $Z$-matrix. We illustrate our result by giving a suitable example of singular hidden $Z$-matrix. We show that linear complementarity problem with hidden $Z$-matrix has unique non-degenerate solution under some assumptions. Finally we show that a linear complementarity problem with hidden $Z$-matrix can be solved using linear programming problem.

\section{Preliminaries}
\noindent We denote the $n$ dimensional real space by $R^n.$ $R^n_+$ denotes the nonnegative orthant of $R^n.$ We consider vectors and matrices with real entries. Any vector $x\in R^{n}$ is a column vector and  $x^{T}$ denotes the row transpose of $x.$ $e$ denotes the vector of all $1.$ If $A$ is a matrix of order $n,$ $\al \subseteq \{1, 2, \cdots, n\}$ and $\bar{\al} \subseteq \{1, 2, \cdots, n\} \setminus \al$ then $A_{\al \bar{\al}}$ denotes the submatrix of $A$ consisting of only the rows and columns of $A$ whose indices are in $\al$ and $\bar{\al}$ respectively. $A_{\al \al}$ is called a principal submatrix of A and det$(A_{\al \al})$ is called a principal minor of $A.$ Given a matrix $A \in R^{n \times n}$ and a vector $q \in R^n,$ we define the feasible set FEA$(q, A)$ $= \{z \in R^n : z \geq 0, q + Az \geq 0\}$ and the solution set of LCP$(q, A)$ by SOL$(q, A)$ $=\{z \in \text{FEA}(q, A) : z^T(q + Az) = 0\}.$

The concept of principal pivot transform \cite{neogy1} plays an important role in this context. Some of the matrix classes are invariant under the principal pivot transform. For details see \cite{das}.

We state the results of two person matrix games in linear system with complementary conditions due to von Neumann \cite{von} and Kaplansky \cite{kaplansky}. The results say that there exist $x^\star \in R^m, y^\star \in R^n$ and $v \in R$ such that 
\begin{center}
	$\sum_{i=1}^{m} x^{\star}_i a_{ij} \leq v, \; \forall\; j = 1, 2, \cdots, n,$\\
	$\sum_{j=1}^{n} y^{\star}_j a_{ij} \geq v, \; \forall\; i = 1, 2, \cdots, m.$
\end{center}
The strategies $(x^{\star}, y^{\star})$ are said to be optimal strategies for player I and player II and $v$ is said to be minimax value of game.
We write $v(A) > 0$ to denote the value of the game corresponding to the payoff matrix $A.$ The value of the game $v(A)$ is positive(nonnegative) if there exists a $0 \neq x \geq 0$ such that $Ax > 0\;(Ax \geq 0).$ Similarly, $v(A)$ is negative(nonpositive) if there exists a $0 \neq y \geq 0$ such that $y^TA < 0\;(y^TA \leq 0).$ The value of the matrix game corresponding to payoff matrix $A \in R^{n\times n}$ is preserved in all its PPTs. 
Now we give the definitions of some matrix classes which will be required in the next section.

A matrix $A \in R^{n \times n}$ is said to be\\
	 $-$ $P\,(P_{0})$-matrix if all its principal minors are positive (nonnegative). \\
	 $-$ almost $P\,(P_{0})$-matrix if all its principal minors upto order $(n-1)$ are positive (nonnegative) and det$(A) < 0.$\\
	 $-$ $E_0\,(E)$ if $0 \neq x \geq 0 \implies x_k > 0 \; \text{and} \; (Ax)_k \geq 0\,((Ax)_k > 0)$ for some $k$.\\
	 $-$ $S$-matrix \cite{pang} if there exists a vector $x>0$ such that $Ax>0$ and $\bar{S}$-matrix if all its principal submatrices are $S$-matrix.\\
	 $-$ $Z$-matrix if off-diagonal elements are all non-positive and $K\,(K_0)$-matrix if it is a $Z$-matrix as well as $P\,(P_0)$-matrix.\\
	 $-$ $N$-matrix if all its principal minors are negative. An $N$-matrix is called an $N$-matrix of the first category if it contains at least one positive entry otherwise it is called an $N$-matrix of the second category.\\
	$-$ $Q$-matrix if for every $q,$ LCP$(q, A)$ has at least one solution. \\
	
Now we give some theorems which will be required for discussion in the next section.
\begin{theorem} \cite{cps} \label{hiddenz}
	Let $A \in R^{n \times n}$ be a hidden $Z$-matrix. Then for any two $Z$-matrices satisfying $AX = Y$ and $r^TX + s^TY > 0$ for some $r, s \in R^n_+.$ Then
	\begin{enumerate}[(i)]
		\item $X$ is nonsingular and
		\item there exists an index set $\alpha \subseteq \{1, 2, \cdots, n\}$ such that the matrix
			$\left[\begin{array}{rr} 
			X_{\al \al} & X_{\al \bar{\al}} \\
			Y_{\bar{\al} \al} & Y_{\bar{\al} \bar{\al}} \\
			\end{array}\right]$ is a $K$-matrix.
	\end{enumerate}  
\end{theorem}
\begin{theorem} \cite{pang} \label{pangps}
	Let $A \in R^{n \times n}$ be a hidden $Z$-matrix. Then $A$ is a $P$-matrix if and only if $A$ is an $S$-matrix.
\end{theorem}
\begin{theorem} \cite{fiedler} \label{Z}
	Let $A$ be a $Z$-matrix. If $v(A)$ is strictly greater than zero then $A$ is a $P$-matrix.	
\end{theorem}
\begin{theorem} \cite{fiedler} \label{Z1}
	Let $A$ be a $Z$-matrix. If there exists a vector $x$ strictly greater than zero such that $Ax \geq 0,$ then $A$ is a $P_0$-matrix.	
\end{theorem}
\begin{theorem} \cite{cps} \label{identical}
	The classes $E$ and $\bar{S}$ are identical i.e. $E = \bar{S}$.
\end{theorem}
\section{Main results}
We first show that hidden $Z$-matrices are invariant under principal rearrangement. 
\begin{theorem}
	Suppose $A \in R^{n \times n}$ is a hidden $Z$-matrix. Then $PAP^T$ is a hidden $Z$-matrix for any permutation matrix $P.$
\end{theorem} 
\begin{proof}
	Let $A$ be a hidden $Z$-matrix. Then by the definition of hidden $Z$-matrix there exist two $Z$-matrices $X, Y$ and two nonnegative vectors $r, s$ such that 
		$AX = Y$ and
		$r^TX + s^TY > 0.$
Now for any permutation matrix $P,$ $P^{-1}PAP^T(P^T)^{-1}X = Y.$ Thus, $(PAP^T)(P^{T})^{-1}X = PY.$ Therefore $(PAP^T)(P^{T})^{-1}XP^T = PYP^T.$ Now letting $X_1 = (P^{T})^{-1}XP^T$ and $Y_1 = PYP^T,$ we get $(PAP^T)X_1 = Y_1.$ It is easy to show that $X_1$ and $Y_1$ are $Z$-matrices. Now for $r_1, s_1 \in R^{n}_{+},$ $r_1^TX_1 + s_1^TY_1 = r^TP^T((P^{T})^{-1}XP^T) + s^TP^{-1}(PYP^T).$ Therefore $r_1^TX_1 + s_1^TY_1 = r^TXP^T + s^TYP^T = (r^TX + s^TY)P^T > 0.$ Thus $PAP^T$ satisfies the definition of hidden $Z$-matrix.
\end{proof}

Fiedler and Pt{\'a}k \cite{fiedler} proved that if there exists a vector $x > 0$ such that $Ax \geq 0$ for a $Z$-matrix $A$, then $A$ is a $P_0$-matrix. We extend this result to hidden $Z$-matrices. 
\begin{theorem} \label{p0}
	Let $A$ be a hidden $Z$-matrix with real entries. Suppose there exists a vector $x > 0$ such that $Ax \geq 0.$ Then $A$ is a $P_0$-matrix.
\end{theorem}
\begin{proof}
	We prove this result by induction method on $n.$ The result is trivially true for $n = 1.$ Consider that the result holds for all matrices of order less than $n.$ Now $A$ is a hidden $Z$-matrix with real entries. Then for some $Z$-matrices $X$ and $Y,$ $AX = Y.$ Then there exists a vector $x > 0$ such that $Ax \geq 0.$ This implies $YX^{-1}x \geq 0$ since $X$ is nonsingular. Let $x_1 = X^{-1}x$ which implies $Xx_1 = x > 0.$ Hence $Yx_1 \geq 0.$ Then by the Theorem \ref{hiddenz}, there exists an index set $\alpha \subseteq \{1, 2, \cdots, n\}$ such that the matrix
		$W =$$\left[\begin{array}{rr} 
		X_{\al \al} & X_{\al \bar{\al}} \\
		Y_{\bar{\al} \al} & Y_{\bar{\al} \bar{\al}} \\
		\end{array}\right]$ is a $K$-matrix
	and $Wx_1 \geq 0.$ This gives $x_1 \geq 0.$ Therefore $X$ is a $K$-matrix and for $x > 0,$ $X^{-1}x > 0$ since $X$ is nonsingular. Therefore $x_1 > 0$ with $Yx_1 \geq 0.$ Then by the Theorem \ref{Z1} of \cite{fiedler}, $Y$ is a $P_0$-matrix. This implies $Y$ is a $K_0$-matrix. Therefore det $A \geq 0.$ Now it is sufficient to prove that for any $\bar{\beta} \subset \{1, 2, \cdots, n\}$, the principal submatrix $A_{\bar{\beta}\bar{\beta}}$ of $A$ is a hidden $Z$-matrix and there exists a $y > 0$ such that $A_{\bar{\beta}\bar{\beta}}y \geq 0.$ Now
		$A_{\bar{\beta}\bar{\beta}}(X_{\bar{\beta}\bar{\beta}} - X_{\bar{\beta}\beta}X_{\beta \beta}^{-1}X_{\beta\bar{\beta}}) = Y_{\bar{\beta}\bar{\beta}} - Y_{\bar{\beta}\beta}X_{\beta \beta}^{-1}X_{\beta \bar{\beta}}$
	which implies
		$A_{\bar{\beta}\bar{\beta}}(X/X_{\beta \beta}) = (M/X_{\beta\beta})$, where $M =$$\left[\begin{array}{rr} 
		X_{\beta \beta} & X_{\beta \bar{\beta}} \\
		Y_{\bar{\beta} \beta} & Y_{\bar{\beta} \bar{\beta}} \\
		\end{array}\right].$
	It is easy to show that $(M/X_{\beta\beta})$ is a $Z$-matrix. Since $(X/X_{\beta\beta})$ is a $K$-matrix, therefore $A_{\bar{\beta}\bar{\beta}}$ is a hidden $Z$-matrix. Consider $x_1=$  $\left[\begin{array}{r} 
	u_{\beta}\\
	u_{\bar{\beta}}
	\end{array}\right].$ Then
	\begin{center}
		$\left[\begin{array}{rr} 
		-X_{\bar{\beta}\beta}X_{\beta \beta}^{-1} & I
		\end{array}\right]$$\left[\begin{array}{rr} 
		X_{\beta \beta}  & X_{\beta \bar{\beta}}\\
		X_{\bar{\beta}\beta} &  X_{\bar{\beta}\bar{\beta}}
		\end{array}\right]$$\left[\begin{array}{r} 
		u_{\beta}\\
		u_{\bar{\beta}}
		\end{array}\right]=$$\left[\begin{array}{rr} 
		0 & X/X_{\beta \beta}
		\end{array}\right]$$\left[\begin{array}{r} 
		u_{\beta}\\
		u_{\bar{\beta}}
		\end{array}\right].$
	\end{center}
	Now as $\left[\begin{array}{rr} 
	-X_{\bar{\beta}\beta}X_{\beta \beta}^{-1} & I
	\end{array}\right]$$\left[\begin{array}{rr} 
	X_{\beta \beta}  & X_{\beta \bar{\beta}}\\
	X_{\bar{\beta}\beta} &  X_{\bar{\beta}\bar{\beta}}
	\end{array}\right]$$\left[\begin{array}{r} 
	u_{\beta}\\
	u_{\bar{\beta}}
	\end{array}\right] > 0,$ then $(X/X_{\beta \beta})u_{\bar{\beta}} >0.$ Now consider $y_{\bar{\beta}}=(X/X_{\beta \beta})u_{\bar{\beta}} >0.$ Again $\left[\begin{array}{rr} 
	0 & M/X_{\beta \beta}
	\end{array}\right]$$\left[\begin{array}{r} 
	u_{\beta}\\
	u_{\bar{\beta}}
	\end{array}\right] \geq 0.$ Therefore $(M/X_{\beta \beta})u_{\bar{\beta}} \geq 0.$ So for $y_{\bar{\beta}}>0,$ $A_{\bar{\beta}\bar{\beta}}y_{\bar{\beta}} \geq 0.$ This implies det $A_{\bar{\beta} \bar{\beta}} \geq 0.$ Therefore $A$ is a $P_0$-matrix.
\end{proof}
\begin{remk}
	The above result may not hold if the condition $x > 0$ is changed to $x \geq 0$. To illustrate our result we consider $A=$  $\left[\begin{array}{rr} 
	-1  & 0\\
	-1 &  2
	\end{array}\right].$ It is easy to show that there exists an $x \geq 0$ such that $Ax \geq 0$.
	Note that $A$ is a hidden $Z$-matrix but not a $P_0$-matrix..
\end{remk}
\begin{corol}
	Let $A$ be a hidden $Z$-matrix with real entries and there exists a vector $x > 0$ such that $Ax \geq 0$. Then every Schur complement in $A$ is a hidden $Z$-matrix as well as $P_0$-matrix.
\end{corol}
\begin{corol}
	Let $A$ be a hidden $Z$-matrix with real entries. If there exists a vector $x > 0$ such that $Ax \geq 0$, then the class of all the linear complementarity problems with the matrix $A$ is NP-complete.
\end{corol}
\begin{proof}
	Suppose $A$ is a hidden $Z$-matrix with real entries and there exists a vector $x > 0$ such that $Ax \geq 0.$ Then it follows from Theorem \ref{p0}, $A$ is a $P_0$-matrix. Now by 3.4 of \cite{kojima} the class of LCP$(q, A)$ is NP-complete.
\end{proof}
Consider a singular matrix $A=$ $\left[\begin{array}{rr} 
1 & 1 \\
1 & 1 
\end{array}\right].$ It is easy to show that $v(A) > 0.$ We show the following result in the context of singular hidden $Z$-matrices.
\begin{theorem}\label{theorem}
	Let $A$ be a singular hidden $Z$-matrix. Then $v(A) \not> 0.$
\end{theorem}
\begin{proof}
	We prove this result by contradiction. Let $A$ be a singular hidden $Z$-matrix and $v(A) > 0.$
	We show that there exists an $ \tilde{x}>0$ such that $A \tilde {x} > 0.$ By definition of value positivity there exists an $x \in R^n_{+}$ such that $Ax > 0.$ Let $\tilde x = x + \epsilon e > 0,$ where $\epsilon > 0.$ Then $A \tilde x = A(x + \epsilon e) = Ax + \epsilon Ae.$ If $Ae\geq 0$, it is enough to choose  any $\epsilon >0.$ If not. Let $a = \min_i (Ax)_i >0$ and $b = \max_i |(Ae)_i|.$ Now choose $\epsilon$ such that $a > \epsilon b.$ This implies $\epsilon < a/b.$ Now for $0 < \epsilon < a/b,$ we can get $ \tilde x = x + \epsilon e$ such that $ \tilde x > 0$ and $A \tilde x > 0.$ Now $A$ is a hidden $Z$-matrix with $v(A) > 0.$ We say that there exists an $\tilde x > 0$ such that $A \tilde x > 0.$ Again as $A$ is hidden $Z$-matrix then for some $Z$-matrices $X$ and $Y,$ $AX = Y.$ Since $X$ is nonsingular by the Theorem \ref{hiddenz}, then $YX^{-1}\tilde{x} > 0.$ Let $\tilde{x}_1 = X^{-1}\tilde{x}.$ Then $Y \tilde{x}_1 > 0$ and $X \tilde{x}_1 > 0.$ Then by the Theorem \ref{hiddenz}, there exists an index set $\alpha \subseteq \{1, 2, \cdots, n\}$ such that the matrix
		$W =$$\left[\begin{array}{rr} 
		X_{\al \al} & X_{\al \bar{\al}} \\
		Y_{\bar{\al} \al} & Y_{\bar{\al} \bar{\al}} \\
		\end{array}\right]$ is a $K$-matrix
	and $W \tilde{x}_1 > 0.$ Let $\tilde{x}_2 = W \tilde{x}_1 > 0.$ Then $\tilde{x}_1 = W^{-1} \tilde{x}_2 > 0$ since $W^{-1} \geq 0.$ Hence for any $\tilde{x}_1 \geq 0,$ $X \tilde{x}_1 > 0$ and $Y \tilde{x}_1 > 0.$ Therefore $v(X) > 0$ and $v(Y) > 0.$ Now as $X$ and $Y$ are $Z$-matrices then by the Theorem \ref{Z}, $X$ and $Y$ are $P$-matrices. Thus we have $\text{det}\, Y > 0$ and $\text{det}\, X^{-1} > 0.$ Therefore det$A >0$ which contradicts the fact that $A$ is singular matrix.   
\end{proof}
We consider a singular hidden $Z$-matrix $A$ to show that $v(A) \not> 0$ with the help of the Theorem \ref{theorem}.
\begin{examp}
 	Let $A=$ $\left[\begin{array}{rrr} 
 	1 & 1 & 0\\
   -1 & -1 & 0\\
 	0 & 0 & 1\\
 	\end{array}\right].$ Note that $A$ is singular matrix. Now $A$ is hidden a $Z$-matrix with 
 	$X=$ $\left[\begin{array}{rrr} 
 	2 & -1 & 0\\
 	-1 & 1 & 0\\
 	0 & -1 & 3\\
 	\end{array}\right]$ 
    and
 	$Y=$ $\left[\begin{array}{rrr} 
 	1 & 0 & 0\\
 	-1 & 0 & 0\\
 	0 & -1 & 3\\
 	\end{array}\right].$ Take $r=$ $\left[\begin{array}{r} 
 	1.6\\
 	4\\
 	2\\
 	\end{array}\right]$ and $s=$ $\left[\begin{array}{r} 
 	4\\
 	0\\
 	0.1\\
 	\end{array}\right],$ then it is easy to show $r^TX + s^TY > 0.$ Then by Theorem \ref{maintheorem}, $v(A) \not> 0$.
 \end{examp}

 Neogy et al. \cite{dubey} show that if $A$ is a hidden $Z$-matrix with $v(A) > 0$ and some additional assumptions, then $A$ is an $E_0$-matrix. In this paper we show that a hidden $Z$-matrix with $v(A)>0$ is a $P$-matrix.

\begin{theorem} \label{maintheorem}
	Let $A$ be a hidden $Z$-matrix and $v(A) > 0.$ Then $A$ is a $P$-matrix.
\end{theorem}
\begin{proof}
 	Let $A \in R^{n \times n}$ with $v(A) > 0.$ Then there exists an $x\in R^n_{+}$ such that $Ax>0.$ In view of Theorem \ref{theorem}  $\exists \ \tilde{x}>0 $ such that $A\tilde{x}>0.$ Then by Theorem \ref{pangps}, $A$ is a $P$-matrix.
\end{proof}
\begin{remk}
For a hidden $Z$-matrix $A$ with $v(A)>0,$ LCP$(q,A)$ is processable by criss-cross method \cite{jana}.
\end{remk}
Now we illustrate our result considering the following example.
\begin{examp}
	Let $A=$ $\left[\begin{array}{rrr} 
	1 & 2 & 0\\
	0 & 1 & 0\\
	-1 & 0 & 1\\
	\end{array}\right].$ For $X=$ $\left[\begin{array}{rrr} 
	1 & -2 & 0\\
	0 & 1 & 0\\
	-1 & -2 & 1\\
	\end{array}\right]$ and \\
	$Y=$ $\left[\begin{array}{rrr} 
	1 & 0 & 0\\
	0 & 1 & 0\\
	-2 & 0 & 1\\
	\end{array}\right],$ $r=$ $\left[\begin{array}{r} 
	3\\
	8\\
	0\\
	\end{array}\right]$ and $s=$ $\left[\begin{array}{r} 
	0\\
	0\\
	1\\
	\end{array}\right],$ we obtain $r^TX + s^TY >0.$ Hence $A$ is a hidden $Z$-matrix. For $x=$ $\left[\begin{array}{r} 
	1\\
	4\\
	5\\
	\end{array}\right],$   $v(A)>0.$ Therefore by Theorem \ref{maintheorem}, the system LCP$(q, A)$ has a unique solution for each $q \in R^n.$
\end{examp}
Here we propose a method to find whether a hidden $Z$-matrix $A$ is $P$-matrix or not.\\
\textbf{Algorithm:}\\ 
\text{Step I:} Choose $\epsilon,\; \delta > 0.$ Consider the following linear programming problem
\be \label{lp}
\begin{array}{ll}
\text{minimize} & s \\
\text{subject to}& Ax - se \geq 0\\
& x \geq \delta e\\
& s \geq \epsilon
\end{array}
\ee
If solution of the linear programming problem exists then by Theorem \ref{maintheorem}, $A$ is a $P$-matrix, else go to Step II.
\vsp
\noindent \text{Step II:} Choose $\epsilon = 0,\; \delta > 0$ and consider the linear programming problem (\ref{lp}). If the solution of the linear programming problem exists then by Theorem \ref{p0}, $A$ is a $P_0$-matrix, else decision is inconclusive. 
\vsp
Note that all $2 \times 2$ $P$-matrices are hidden $Z$ but in general there are $P$-matrices which are not hidden $Z$ \cite{pang}. Now we show the condition under which a $P$-matrix is a hidden $Z$-matrix. For this purpose we consider the definition of $D$-matrix.
\begin{defn}
A matrix $A \in R^{n \times n}$ is said to be type $D$ \cite{markham} if there exist some real numbers $\{\al_i\}_{i = 1}^{n}$ with $\al_n > \al_{n-1}> \cdots > \al_1,$ such that
\begin{center}
	$a_{ij} = \begin{cases}
	\al_i \ \text{if} \ i \leq j;\\
	\al_j \ \text{if} \ i > j.
	\end{cases}$
\end{center} 
\end{defn}
\begin{theorem}
Suppose $A$ is positive type $D$-matrix. Then $A$ is a hidden $Z$-matrix.
\end{theorem}
\begin{proof}
 Suppose $A$ is positive type $D$-matrix. It is easy to show that positive type $D$-matrices are $P$-matrices. Then $A$ is nonsingular and $A^{-1}$ is $Z$-matrix as shown in \cite{markham} which in turn implies $A^{-1}$ is hidden $Z$-matrix. Now $A$ is a PPT of $A^{-1}$ and PPT of a hidden $Z$-matrix is hidden $Z$ \cite{dubey}. Therefore $A$ is a hidden $Z$-matrix. 
\end{proof}

It is known that inverse of an almost $P$-matrix is an $N$-matrix. To illustrate our result we consider $A=$  $\left[\begin{array}{rr} 
1  & 2\\
1 &  1
\end{array}\right].$ It is easy to show that $A$ is an almost $P$-matrix and $A^{-1}$ is an $N$ of first category. For further details see \cite{neogy2}. Now we prove the following theorem.
 \begin{theorem}
	Let $A$ be a hidden $Z$-matrix with real entries. If $A$ is an almost $P$-matrix then $A^{-1}$ is an $N$-matrix of second category. 
\end{theorem}
\begin{proof}
	Let $A$ be a hidden $Z$-matrix with real entries. If $A$ is an almost $P$-matrix then $A^{-1}$ is an $N$-matrix. Suppose $A^{-1}$ is an $N$-matrix of first category. Then by \cite{olech}, $A^{-1}$ is a $Q$-matrix. Therefore by Theorem \ref{maintheorem} we arrive at a contradiction. This implies $A^{-1}$ is an $N$-matrix of second category. 
\end{proof}

\begin{theorem}
	Let $A$ be a hidden $Z$-matrix with real entries. Assume that $A$ is an $E_0$-matrix and every feasible basis of FEA$(q, A)$ is non-degenerate. Let LCP$(q,A)$ have a solution. Then the problem has a unique non-degenerate solution.
\end{theorem}
\begin{proof}
	Suppose $A$ is a hidden $Z$-matrix with real entries then there exist two $Z$-matrices $X, \, Y$ with two nonnegative vectors $r, s$ such that \begin{center}
		$AX = Y,$\\
		$r^TX + s^TY > 0$.
	\end{center}
    \noindent Consider $A(\epsilon) = A+ \epsilon I$ for all $\epsilon \in (0,l),$ where \[ l=  \frac{\min_{i} (r^TX + s^TY)} {\max_{i} \mid(s^TX)\mid}. \] As $X$ is nonsingular by the Theorem \ref{hiddenz}, it is clear that $s^TX \neq 0$. Now $(A + \epsilon I)X = AX + \epsilon X = Y + \epsilon X.$ Note that $Y + \epsilon X$ is a $Z$-matrix. Again $r^TX + s^T(Y + \epsilon X) > 0$ by the choice of $\epsilon.$ Hence $A(\epsilon)$ is a hidden $Z$-matrix. Note that $A$ is an $E_0$-matrix. It is easy show that $(A + \epsilon I)$ is an $E$-matrix.   
	Let $(-A._k, I._{\bar{k}}),$ where $k \subseteq \{1,2, \cdots, n\}$ and $\bar{k} \subseteq \{1,2,\cdots n\} \setminus k$ denote a basis. By our assumption, $z_k = -(A_{kk})^{-1}q_k >0$, $w_{\bar{k}} = q_{\bar{k}} - A_{\bar{k}k}(A_{kk})^{-1}q_k > 0$. For sufficiently small $\epsilon \in (0,l)$, $A(\epsilon)_{kk}$ is nonsingular and $z_{k}'=  -(A(\epsilon)_{kk})^{-1}q_k > 0 \ \text{and} \ w_{\bar{k}}' = q_{\bar{k}} - A(\epsilon)_{\bar{k}k}(A(\epsilon)_{kk})^{-1}q_k > 0$. Therefore $z'=(z_{k}',0), \ w'=(0,w_{\bar{k}}')$ is a non-degenerate solution to $LCP(q,A(\epsilon))$. Assume that $(-A._p, I._{\bar{p}})$ denotes another complementary feasible basis for LCP$(q,A),$ where $k \neq p \subseteq \{1,2, \cdots n\}$ and $\bar{p} \subseteq \{1,2,\cdots n\} \setminus p.$ Hence $(w'',z'')$ is another non-degenerate solution to the LCP$(q, A(\epsilon)).$ Note that by the Theorem \ref{identical}, $A(\epsilon)$ is an $\bar{S}$-matrix. Then by the property $2$ of \cite{chusemimonotone}, it contradicts that LCP$(q, A(\epsilon))$ has unique solution as $A(\epsilon)$ is a $P$-matrix. Therefore LCP$(q,A)$ has a unique non-degenerate solution.
\end{proof}
Now we show some sufficient conditions under which a principal submatrix of a hidden $Z$-matrix will be hidden $Z.$
\begin{theorem}
	Let $A$ be a hidden $Z$-matrix with real entries such that $AX = Y$ and $r^TX + s^TY > 0$ where $X, Y$ are $Z$-matrix and $r, s \in R^n_+.$ If there exists an index set $\alpha \subset \{1,2,\cdots, n\}$ such that $W=$$\left[\begin{array}{rr} 
	X_{\al \al} & X_{\al \bar{\al}} \\
	Y_{\bar{\al} \al} & Y_{\bar{\al} \bar{\al}} \\
	\end{array}\right]$ and $\bar{W}=$$\left[\begin{array}{rr} 
	Y_{\al \al} & Y_{\al \bar{\al}} \\
	X_{\bar{\al} \al} & X_{\bar{\al} \bar{\al}} \\
	\end{array}\right]$ are $E$-matrices. Then $A_{\al \al}$ and $A_{\bar{\al} \bar{\al}}$ are hidden $Z$-matrices. 
\end{theorem}
\begin{proof}
	Note that $A$ is a hidden $Z$-matrix with real entries then there exist two $Z$-matrices $X, \, Y$ with two nonnegative vectors $r, s$ such that \begin{center}
		$AX = Y,$\\
		$r^TX + s^TY > 0.$
	\end{center}
	This implies for an index set $\alpha \subset \{1,2,\cdots n\}$ such that $W=$$\left[\begin{array}{rr} 
	X_{\al \al} & X_{\al \bar{\al}} \\
	Y_{\bar{\al} \al} & Y_{\bar{\al} \bar{\al}} \\
	\end{array}\right]$ and $\bar{W}=$$\left[\begin{array}{rr} 
	Y_{\al \al} & Y_{\al \bar{\al}} \\
	X_{\bar{\al} \al} & X_{\bar{\al} \bar{\al}} \\
	\end{array}\right]$ are $Z$-matrices as well as $E$-matrices. Therefore $W/X_{\al \al}$, $\bar{W}/X_{\bar{\al} \bar{\al}}$ are $K$-matrices. Also note that $X/X_{\al \al}$, $X/X_{\bar{\al} \bar{\al}}$ are $Z$-matrices. This implies that $A_{\al \al}$ is a $\text{hidden}\  Z$-matrix with $X/X_{\bar{\al} \bar{\al}}, \bar{W}/X_{\bar{\al} \bar{\al}}$ are $Z$-matrices such that $A_{\al \al}(X/X_{\bar{\al} \bar{\al}})=\bar{W}/X_{\bar{\al} \bar{\al}}$. Similarly the principal submatrix $A_{\bar{\al} \bar{\al}}$ is a $\text{hidden}\ Z$-matrix with $X/X_{\al \al},$ $W/X_{\al \al}$ are $Z$-matrices such that $A_{\bar{\al}  \bar{\al}}(X/X_{\al \al})=W/X_{\al \al}$. 
\end{proof}
\begin{theorem}\label{complth}
	Let $A$ be a hidden $Z$-matrix with real entries such that $AX = Y$ and $r^TX + s^TY > 0$ where $X, Y$ are $Z$-matrix and $r, s \in R^n_+.$ If there exists an index set $\alpha = \{1,2,\cdots, n\}$ such that $W=$$\left[\begin{array}{rr} 
	X_{\al \al} & X_{\al \bar{\al}} \\
	Y_{\bar{\al} \al} & Y_{\bar{\al} \bar{\al}} \\
	\end{array}\right]$ and $\bar{W}=$$\left[\begin{array}{rr} 
	Y_{\al \al} & Y_{\al \bar{\al}} \\
	X_{\bar{\al} \al} & X_{\bar{\al} \bar{\al}} \\
	\end{array}\right]$ are $E$-matrices. Then $A$ is a completely hidden $Z$-matrix.  
\end{theorem}
\begin{proof}
	Note that $A$ is a hidden $Z$-matrix with real entries then there exist two $Z$-matrices $X, \, Y$ with two nonnegative vectors $r, s$ such that \begin{center}
		$AX = Y,$\\
		$r^TX + s^TY > 0$.
	\end{center}
	Hence for an index set $\alpha = \{1,2,\cdots, n\}$ such that $W=$$\left[\begin{array}{rr} 
	X_{\al \al} & X_{\al \bar{\al}} \\
	Y_{\bar{\al} \al} & Y_{\bar{\al} \bar{\al}} \\
	\end{array}\right] = X$ and $\bar{W}=$$\left[\begin{array}{rr} 
	Y_{\al \al} & Y_{\al \bar{\al}} \\
	X_{\bar{\al} \al} & X_{\bar{\al} \bar{\al}} \\
	\end{array}\right]  = Y$ are $Z$-matrices as well as $E$-matrices. This implies that for any $\beta \subset \{1, 2, \cdots, n\}$, $X/X_{\beta \beta}$ and $X/X_{\bar{\beta} \bar{\beta}}$ are $K$-matrices. Then the principal submatrix $A_{\beta \beta}$ of $A$ is a $\text{hidden} \ Z$-matrix with $X/X_{\bar{\beta} \bar{\beta}}, \bar{M}/X_{\bar{\beta} \bar{\beta}}$ are $Z$-matrices such that $A_{\beta \beta}(X/X_{\bar{\beta} \bar{\beta}})=\bar{M}/X_{\bar{\beta} \bar{\beta}},$ where $\bar{M}=$ $\left[\begin{array}{rr} 
	Y_{\beta \beta} & Y_{\beta \bar{\beta}} \\
	X_{\bar{\beta} \beta} & X_{\bar{\beta} \bar{\beta}} \\
	\end{array}\right].$
\end{proof}
\begin{remk}
	Let $A$ be a hidden $Z$-matrix with real entries such that $AX = Y$ and $r^TX + s^TY > 0$ where $X, Y$ are $Z$-matrix and $r, s \in R^n_+$ and suppose there exists an empty index set $\alpha$ such that $W=$$\left[\begin{array}{rr} 
	X_{\al \al} & X_{\al \bar{\al}} \\
	Y_{\bar{\al} \al} & Y_{\bar{\al} \bar{\al}} \\
	\end{array}\right]$ and $\bar{W}=$$\left[\begin{array}{rr} 
	Y_{\al \al} & Y_{\al \bar{\al}} \\
	X_{\bar{\al} \al} & X_{\bar{\al} \bar{\al}} \\
	\end{array}\right]$ are $E$-matrices. Then $A$ is a completely hidden $Z$-matrix.  
\end{remk}

Now we introduce an alternative linear programming problem to solve  linear complementarity problem with hidden $Z$-matrix.
\begin{theorem}
Let $A$ be a hidden $Z$-matrix with real entries such that $AX = Y$ and $r^TX + s^TY > 0$ where $X, Y$ are $Z$-matrix and $r, s \in R^n_+.$ Then the linear complementarity problem denoted by LCP$(q, A)$ can be written as 
	$$\begin{array}{ll}
	\text{minimize} & (r + A^Ts)^Tz_1 + q^Tz_2\\
	\text{subject to} &  A^Ts + r - A^Tz_2 \geq 0,\\
	& Az_1 + q \geq 0,\\
	& z_1, z_2 \geq 0. 
	\end{array}$$
\end{theorem}
\begin{proof}
To prove our result we consider LCP$(\bar{q}, \bar{A})$ where $\bar{A}=$ $\left[\begin{array}{cc} 
0 & -A^T\\
A & 0 \\
\end{array}\right]$ and $\bar{q}=$ $\left[\begin{array}{c} 
p \\
q \\
\end{array}\right]$ with $p = r+A^Ts.$ By Lemma 1 of \cite{mangasarian} and Lemma 3.3 of \cite{dubey}, LCP$(q, A)$ and LCP$(\bar{q}, \bar{A})$ are equivalent. Assume $\bar{z}=$ $\left[\begin{array}{c} 
z_1 \\
z_2 \\
\end{array}\right]$ be the solution of LCP$(\bar{q}, \bar{A}).$ Note that $\bar{A}$ is a skew symmetric matrix. Now LCP$(\bar{q}, \bar{A})$ can be written as 
$$\begin{array}{ll}
\text{minimize} & \bar{q}^T \bar{z} + \frac{1}{2} \bar{z}^T(\bar{A} + \bar{A}^T)\bar{z}\\
\text{subject to} & \bar{q} + \bar{A}\bar{z} \geq 0,\\
& \bar{z} \geq 0. 
\end{array}$$
Again equivalent quadratic programming problem can be rewritten as 
$$\begin{array}{ll}
\text{minimize} & (r + A^Ts)^Tz_1 + q^Tz_2\\
\text{subject to} & A^Ts + r -A^Tz_2 \geq 0,\\
& Az_1 + q \geq 0,\\
& z_1, z_2 \geq 0. 
\end{array}$$

\end{proof}
 
\section{Conclusion}
In this article, we study the class of hidden $Z$-matrix in the context of linear complementarity problem. We show that linear complementarity problem with hidden $Z$-matrix is processable by Lemke's algorithm as well as criss-cross method. To prove our result we apply the concept of principal pivot transform and game theoretic approach. We establish certain matrix theoretic characterization of hidden $Z$-matrix to show the $P_0$ properties.

\section*{Acknowledgement}
The authors R. Jana and A. Dutta are thankful to the Department of Science and Technology, Govt. of India, INSPIRE Fellowship Scheme for financial support.
\vsp

\bibliographystyle{plain}
\bibliography{bibfile}
\end{document}